\documentclass{amsart}
\usepackage{textcomp}
\usepackage{indentfirst}
\usepackage{amsmath}
\usepackage{amsthm}
\usepackage{amsfonts}
\usepackage{amssymb}
\usepackage{mathrsfs}
\usepackage{stmaryrd}
\usepackage{bbm}
\usepackage[T1]{fontenc}

\newtheorem{thm}{Theorem}[section]
\newtheorem{cor}[thm]{Corollary}
\newtheorem{lem}[thm]{Lemma}
\newtheorem{prop}[thm]{Proposition}
\theoremstyle{definition}

\numberwithin{equation}{section}

\begin{document}

\title{Asymmetric truncated Toeplitz operators equal to the zero operator}
\author[Joanna Jurasik]{Joanna Jurasik}%\\
\author[Bartosz \L anucha]{Bartosz \L anucha}%\\
%Department of Mathematics, Maria Curie-Sklodowska University\\
%20-031 Lublin, Poland\\
%E-mail: bartosz.lanucha@gmail.com}
\address{Department of Mathematics, Maria Curie-Sk\l odowska University, Maria Curie-Sk\l odowska Square 1, 20-031 Lublin, Poland}%
\email{bartosz.lanucha@poczta.umcs.lublin.pl}%
\address{Department of Mathematics, Maria Curie-Sk\l odowska University, Maria Curie-Sk\l odowska Square 1, 20-031 Lublin, Poland}%
\email{asia.blicharz@op.pl}%
\begin{abstract}
Asymmetric truncated Toeplitz operators are compressions of multiplication operators acting between two model spaces. These operators are natural generalizations of truncated Toeplitz operators. In this paper we describe symbols of asymmetric truncated Toeplitz operators equal to the zero operator.
\end{abstract}

\maketitle

\renewcommand{\thefootnote}{}



\renewcommand{\thefootnote}{\arabic{footnote}}
\setcounter{footnote}{0}

\section{Introduction}

Let $H^2$ denote the Hardy space of the unit disk
$\mathbb{D}=\{z\colon|z|<1\}$, that is, the space of functions analytic in $\mathbb{D}$ with square summable Maclaurin coefficients.

Using the boundary values, one can identify $H^2$ with a closed
subspace of $L^2(\partial\mathbb{D})$, the subspace of
functions whose Fourier coefficients with negative indices vanish.
The orthogonal projection $P$ from $L^2(\partial\mathbb{D})$ onto $H^2$, called the Szeg\"{o} projection, is given by
$$Pf(z)=\frac{1}{2\pi}\int_0^{2\pi}\frac{f(e^{it})dt}{1-e^{-it}z},\quad f\in L^2(\partial\mathbb{D}).$$
Note that if $f\in L^1(\partial\mathbb{D})$, then the above integral still defines a function $Pf$ analytic in $\mathbb{D}$.

%As an integral operator, $P$ can be extended
%to $L^1(\partial\mathbb{D})$.

The classical Toeplitz operator $T_{\varphi}$ with symbol
$\varphi\in L^2(\partial\mathbb{D})$ is defined on $H^2$ by
$$T_{\varphi}f=P(\varphi f).$$%,\quad f\in  H^2.$$
It is known that $T_{\varphi}$ is bounded if and only if $\varphi\in
L^{\infty}(\partial\mathbb{D})$. The operator $S=T_z$ is called the unilateral shift and its adjoint $S^{*}=T_{\overline{z}}$ is called the backward shift. We have $Sf(z)=zf(z)$ and
$$S^{*}f(z)=\frac{f(z)-f(0)}{z}.$$

Let $H^{\infty}$ be the algebra of bounded analytic functions on
$\mathbb{D}$ and let $\alpha\in H^{\infty}$ be an arbitrary inner function, that
is, $|\alpha|=1$ a.e. on $\partial\mathbb{D}$.

By the theorem of A. Beurling (see, for example, \cite[Thm.
8.1.1]{r}), every nontrivial, closed $S$-invariant subspace of $H^2$ can be expressed as $\alpha H^2$ for some inner function $\alpha$. Consequently, every nontrivial, closed $S^{*}$-invariant subspace of $H^2$ is of the form
$$K_{\alpha}=H^2\ominus \alpha H^2$$
with $\alpha$ inner. The space $K_{\alpha}$ is called the model space corresponding to $\alpha$.

The kernel function
\begin{equation}\label{kerker}
k_{w}^{\alpha}(z)=\frac{1-\overline{\alpha(w)}\alpha(z)}{1-\overline{w}z},\quad w, z\in\mathbb{D},
\end{equation}
is a reproducing kernel for the model space $K_{\alpha}$, i.e., for each
$f\in K_{\alpha}$ and $w\in\mathbb{D}$,
$$f(w)=\langle f, k_{w}^{\alpha}\rangle$$
($\langle
\cdot,\cdot\rangle$ being the usual integral inner product). Observe that
$k_{w}^{\alpha}$ is a bounded function for every $w\in\mathbb{D}$. It follows that the subspace
$K_{\alpha}^{\infty}=K_{\alpha}\cap H^{\infty}$ is dense in $K_{\alpha}$. If $\alpha(w)=0$, then $k_{w}^{\alpha}=k_{w}$, where $k_w$ is the Szeg\"{o} kernel given by $k_w(z)=(1-\overline{w}z)^{-1}$.

The function $\alpha$ is said to have a nontangential limit at $\eta\in\partial\mathbb{D}$ if there exists $\alpha(\eta)$ such that $\alpha(z)$ tends to $\alpha(\eta)$ as $z\in\mathbb{D}$ tends to $\eta$ nontangentially (with $|z-\eta|\leq C(1-|z|)$ for some fixed $C>0$). We say that $\alpha$ has an angular derivative in the sense of Carath\'{e}odory (an ADC) at $\eta\in\partial\mathbb{D}$ if both $\alpha$ and $\alpha'$ have nontangential limits at $\eta$ and $|\alpha(\eta)|=1$ (for more details see \cite[pp. 33--37]{bros}). P. R. Ahern and D. N. Clark proved in \cite{clark,cclark}, that $\alpha$ has an ADC at $\eta\in\partial\mathbb{D}$ if and only if every $f\in K_{\alpha}$ has a nontangential limit $f(\eta)$ at $\eta$. If $\alpha$ has an ADC at $\eta$ and $w$ tends to $\eta$ nontangentially, then the reproducing kernels $k_w^{\alpha}$ tend in norm to the function $k_{\eta}^{\alpha}\in K_{\alpha}$ given by \eqref{kerker} with $\eta$ in place of $w$. Moreover, $f(\eta)=\langle f, k_{\eta}^{\alpha}\rangle$ for all $f\in K_{\alpha}$.

Let $P_{\alpha}$ denote the orthogonal projection
 from $L^2(\partial\mathbb{D})$ onto $K_{\alpha}$. Then
  $$P_{\alpha}f(z)=\langle f,k_{z}^{\alpha}\rangle,\quad f\in
L^2(\partial\mathbb{D}),\ z\in\mathbb{D}.$$
%=\frac{1}{2\pi}\int_0^{2\pi}f(e^{it})\frac{1-\alpha(z)\overline{\alpha(e^{it})}}{1-ze^{-it}}dt
Just like with the Szeg\"{o} projection,  $P_{\alpha}f$ is a function analytic in $\mathbb{D}$ for all $f\in L^1(\partial\mathbb{D})$.

%The model space $K_{\alpha}$ is equipped with a natural conjugation
%(antilinear, isometric involution) $C_{\alpha}\colon
%K_{\alpha}\rightarrow K_{\alpha}$, defined in terms of the boundary
%values by
%\begin{equation}\label{numerek3}
%C_{\alpha}f(z)=\alpha(z)\overline{z}\overline{f(z)},\quad |z|=1
%\end{equation}
%(see \cite{s} Section 2.3, for details). Actually, it can be
%verified that \eqref{numerek3} defines an antilinear and isometric
%involution on $L^2(\partial\mathbb{D})$, which maps $\alpha H^2$
%onto $\overline{H_0^2}$, $\overline{H_0^2}$ onto $\alpha H^2$, and
%preserves $K_{\alpha}$. For $f\in K_{\alpha}$ we will write $\widetilde{f}$ in place of
%$C_{\alpha}f$ when no confusion can arise. A short calculation shows that the conjugate kernel
%$\widetilde{k}_{w}^{\alpha}$ is given by
%$$\widetilde{k}_{w}^{\alpha}(z)=\frac{\alpha(z)-\alpha(w)}{z-w}.$$
%If $\eta\in\partial\mathbb{D}$ and $k_{\eta}^{\alpha}\in K_{\alpha}$, then
%$$\widetilde{k}_{\eta}^{\alpha}(z)=\frac{\alpha(z)-\alpha(\eta)}{z-\eta}=\alpha(\eta)\overline{\eta}k_{\eta}^{\alpha}(z).$$

A truncated Toeplitz operator with a symbol $\varphi\in
L^2(\partial\mathbb{D})$ is the operator $A_{\varphi}^{\alpha}$ defined on the model space $K_{\alpha}$ by
$$A_{\varphi}^{\alpha}f=P_{\alpha}(\varphi f).$$%,\quad f\in K_{\alpha}.$$
Densely defined on bounded functions, the operator $A_{\varphi}^{\alpha}$ can be seen as a compression to
$K_{\alpha}$ of the classical Toeplitz operator $T_{\varphi}$ on
$H^2$.

%However, despite similar definitions, truncated Toeplitz
%operators differ form the classical ones in many ways. For example, $T_{\varphi}=0$ if and only if $\varphi=0$, but $A_{\varphi}^{\alpha}=0$ if and only if %$\varphi\in \overline{\alpha H^2}+\alpha H^2$ (see \cite{s}, Thm. 3.1). Moreover, unbounded symbols can produce bounded truncated Toeplitz operators
%and there are bounded truncated Toeplitz operators for which no
%bounded symbol exists (for details, see \cite{bar}).

The study of truncated Toeplitz operators as a class began in 2007
with D. Sarason's paper \cite{s}. In spite of similar definitions, there are many differences between truncated Toeplitz operators and the classical ones. One of the first results from \cite{s} states that, unlike in the classical case, a truncated Toeplitz operator is not uniquely determined by its symbol. More precisely, $A_{\varphi}^{\alpha} = 0$ if and only if $\varphi\in\overline{\alpha H^2}+\alpha H^2$ (\cite[Thm. 3.1]{s}). As a consequence, unbounded symbols can produce bounded truncated Toeplitz operators. Moreover, there exist bounded truncated Toeplitz operators for which no bounded symbols exist (see \cite{bar}). For more interesting results see \cite{si,bros,gar,gar2,gar3}.

% The set of all bounded truncated Toeplitz operators on $K_u$
%will be denoted by $\mathscr{T}(K_u)$. Truncated Toeplitz operators
%have been studied in, for example \cite{bar}, \cite{si},
%\cite{gp}-\cite{gar3}, \cite{s}.

%It is known that every $A\in\mathscr{T}(K_u)$ is
%$C$-symmetric, that is,
%\begin{equation}\label{dens}
%CA=A^{*}C.
%\end{equation}

Recently, the authors in \cite{ptak} and \cite{part} introduced a generalization of truncated Toeplitz operators, the so-called asymmetric truncated
Toeplitz operators. Let $\alpha$, $\beta$ be two inner functions and let $\varphi\in L^2(\partial\mathbb{D})$. An asymmetric
truncated Toeplitz operator $A_{\varphi}^{\alpha,\beta}$ is the
operator from $K_{\alpha}$ into $K_{\beta}$ given by
$$A_{\varphi}^{\alpha,\beta}f=P_{\beta}(\varphi f),\quad f\in K_{\alpha}.$$
The operator $A_{\varphi}^{\alpha,\beta}$ is densely defined. Clearly, $A_{\varphi}^{\alpha,\alpha}=A_{\varphi}^{\alpha}$.

We denote
$$\mathcal{T}(\alpha,\beta)=\{A_{\varphi}^{\alpha,\beta}\ \colon\ \varphi\in
L^2(\partial\mathbb{D})\ \mathrm{and}\ A_{\varphi}^{\alpha,\beta}\
\mathrm{is\ bounded}\}$$
and $\mathcal{T}(\alpha)=\mathcal{T}(\alpha,\alpha)$.

The purpose of this paper is to describe when an operator from $\mathcal{T}(\alpha,\beta)$ is equal to the zero operator. The description is given in terms of the symbol of the operator. This was done in \cite{ptak} and \cite{part} for the case when $\beta$ divides $\alpha$, that is, when $\alpha/\beta$ is an inner function. It was proved in \cite{ptak} and \cite{part} that $A_{\varphi}^{\alpha,\beta}=0$ if and only if $\varphi\in \overline{\alpha H^2}+\beta H^2$. Here we show that this is true for all inner functions $\alpha$ and $\beta$. We also give some examples of rank-one asymmetric truncated Toeplitz operators.

\section{Main result}

In this section we prove the following.

\begin{thm}\label{zero}
Let $\alpha,\beta$ be two nonconstant inner functions and let
$A^{\alpha, \beta}_{\varphi}:K_{\alpha}\to K_{\beta}$ be a bounded asymmetric truncated Toeplitz operator with $\varphi\in L^2(\partial\mathbb{D})$. Then $A^{\alpha, \beta}_{\varphi}=0$ if and only if $\varphi\in \overline{\alpha H^2}+\beta H^2$.
\end{thm}

We start with a simple technical lemma.

\begin{lem}\label{le1}
Let $\alpha$, $\beta$ be two arbitrary inner functions. If
\begin{equation}\label{9}
K_{\alpha}\subset \beta H^2,
\end{equation}
then both $\alpha$ and $\beta$ have no zeros in $\mathbb{D}$, or at least one of the functions $\alpha$ or $\beta$ is a constant function.
%then either one of the functions $\alpha, \beta$ is constant, or they both have no zeros in $\mathbb{D}$.
\end{lem}
\begin{proof}
Assume that (\ref{9}) holds. If $\beta(z_0)=0$ for some $z_0\in\mathbb{D}$, then $f(z_0)=0$ for every $f\in K_{\alpha}$. For $f=k_{z_0}^{\alpha}$ we get
\begin{displaymath}
k_{z_0}^{\alpha}(z_0)=\Vert k_{z_0}^{\alpha}\Vert ^2 =
\frac{1-\vert\alpha(z_0)\vert ^2}{1-\vert z_0\vert ^2}=0,
\end{displaymath} which implies that $\vert\alpha(z_0)\vert=1$. By the maximum modulus principle, $\alpha$ is a constant function.
Hence, the inclusion \eqref{9} implies that $\beta$ has no zeros in $\mathbb{D}$, or $\alpha$ is a constant function. But \eqref{9} is equivalent to
\begin{equation*}
K_{\beta}\subset \alpha H^2,
\end{equation*}
and, by the same reasoning, \eqref{9} also implies that $\alpha$ has no zeros in $\mathbb{D}$, or $\beta$ is a constant function. This completes the proof.
\end{proof}

Lemma \ref{le1} can be rephrased as follows. If $\alpha$, $\beta$ are two nonconstant inner functions and at least one of them has a zero in $\mathbb{D}$, then the inclusion $K_{\alpha}\subset \beta H^2$ does not hold. This allows us to prove the following version of Theorem \ref{zero}.

%In particular, this is the case when $\alpha$ and $\beta$ are two nonconstant Blaschke products.

%It is easy to verify that if $\varphi\in \overline{\alpha H^2}+\beta H^2$ then $A^{\alpha, \beta}_{\varphi}=0$ (see \cite{ptak} or the proof below). The authors in \cite{ptak} prove the equivalence of this two conditions in the case when $\beta\leq\alpha$ ($\beta$ divides $\alpha$, i.e., $\alpha=\beta h$ for some $h\in H^2$). Here we give a proof of that equivalence without assuming that $\beta\leq\alpha$. This was also done in \cite{part}. The authors in \cite{part} prove the fact for the case of asymmetric truncated Toeplitz operators on the half-plane.

\begin{prop}\label{zero1}
Let $\alpha$, $\beta$ be two nonconstant inner functions such that each of them has a zero in $\mathbb{D}$ and let
$A^{\alpha, \beta}_{\varphi}:K_{\alpha}\to K_{\beta}$ be a bounded asymmetric truncated Toeplitz operator with $\varphi\in L^2(\partial\mathbb{D})$. Then $A^{\alpha, \beta}_{\varphi}=0$ if and only if $\varphi\in \overline{\alpha H^2}+\beta H^2$.
\end{prop}
\begin{proof}
The fact that $\varphi\in \overline{\alpha H^2}+\beta H^2$ implies $A^{\alpha, \beta}_{\varphi}=0$ was proved in \cite[Thm. 4.3]{ptak}. For the convenience of the reader we repeat the reasoning from \cite{ptak}.

Assume that $\varphi= \overline{\alpha h}_1+\beta h_2$ with $h_1, h_2\in H^2$. Then, for every $f\in K_{\alpha}^{\infty}$,
\begin{displaymath}
A^{\alpha, \beta}_{\varphi} f=P_{\beta}(\overline{\alpha h}_1f+\beta h_2f)=P_{\beta}(\overline{\alpha h}_1f).
\end{displaymath}
 Since $f\perp \alpha H^2$, we see that $ \overline{\alpha h}_1f\perp H^2$ and $P_{\beta}\left(\overline{\alpha h}_1f\right)=0$. The density of $K_{\alpha}^{\infty}$ implies that $A^{\alpha, \beta}_{\varphi}=0$. Note that this part of the proof does not depend on the existence of zeros of $\alpha$ and $\beta$.

Let us now assume that $A^{\alpha, \beta}_{\varphi}=0$. By the first part of the proof, we can also assume that $\varphi=\overline{\chi}+\psi$ for some $\chi\in K_{\alpha}$, $\psi\in K_{\beta}$. Let $z_0\in\mathbb{D}$ be a zero of $\alpha$. Then $k_{z_0}^{\alpha}=k_{z_0}$ and
\begin{displaymath}
\begin{split}
A^{\alpha, \beta}_{\overline{\chi}}k_{z_0}^{\alpha}&=P_{\beta}(\overline{\chi} k_{z_0})\\
&=P_{\beta}\left(\overline{z}\frac{\overline{\chi(z)}-\overline{\chi(z_0)}}{\overline{z}-\overline{z_0}}+\overline{\chi(z_0)}k_{z_0}\right)\\ &=\overline{\chi(z_0)}k_{z_0}^{\beta},
\end{split}
\end{displaymath}
because the quotient $(\chi(z)-\chi(z_0))/(z-z_0)$ belongs to $K_{\alpha}$ (see \cite[Subsection 2.6]{s}).

Hence,
\begin{displaymath}
\begin{split}
0&=A_{\varphi}^{\alpha, \beta}k_{z_0}^{\alpha}=A_{\overline{\chi}+\psi}^{\alpha, \beta}k_{z_0}^{\alpha}\\ &= \overline{\chi(z_0)}k_{z_0}^{\beta}+A_{{\psi}}^{\alpha, \beta}k_{z_0}^{\alpha}= P_{\beta}\left[(\overline{\chi(z_0)}+\psi )k_{z_0}\right],
\end{split}
\end{displaymath}
which means that
$$(\overline{\chi(z_0)}+\psi )k_{z_0}\in \beta H^2$$
and, consequently,
\begin{equation}\label{10}
\overline{\chi(z_0)}+\psi\in \beta H^2.
\end{equation}

On the other hand (\cite[Lem. 3.2]{ptak}),
\begin{displaymath}
A_{\overline{\psi}+\chi}^{\beta,\alpha}=\left(A_{\overline{\chi}+\psi}^{\alpha, \beta}\right)^{*}=0,
\end{displaymath}
and a similar reasoning can be used to show that if $\beta(w_0)=0$, $w_0\in \mathbb{D}$, then
\begin{equation}\label{11}
\chi+\overline{\psi(w_0)}\in \alpha H^2.
\end{equation}

By \eqref{10}, \eqref{11} and the first part of the proof we get
\begin{displaymath}
A_{\overline{\chi}+\psi(w_0)+\overline{\chi(z_0)}+\psi}^{\alpha, \beta}=0,
\end{displaymath}
and
\begin{displaymath}
A_{\psi(w_0)+\overline{\chi(z_0)}}^{\alpha, \beta}=-A_{\overline{\chi}+\psi}^{\alpha, \beta}=0.
\end{displaymath}
From this,
\begin{displaymath}
P_{\beta}\left [(\psi(w_0)+\overline{\chi(z_0)})f \right]=0
\end{displaymath} for all $f\in K_{\alpha}$.

If $\psi(w_0)+\overline{\chi(z_0)}\neq 0$, then the above means that $P_{\beta}(f)=0$ for all $f\in K_{\alpha}$, that is,
$ K_{\alpha}\subset \beta H^2$. However, by Lemma \ref{le1}, this cannot be the case here. So

\begin{displaymath}
\psi(w_0)+\overline{\chi(z_0)}=0
\end{displaymath} and
\begin{displaymath}
%\begin{split}
\varphi=\overline{\chi}+\psi= \overline{\chi}+\psi(w_0)+\overline{\chi(z_0)}+\psi \in  \overline{\alpha H^2}+\beta H^2.
%\end{split}
\end{displaymath}
%by \eqref{10} and \eqref{11}. This completes the proof.
\end{proof}
%\begin{cor}
%Let $\alpha,\beta$ be two nonconstant Blaschke products and let $A^{\alpha, \beta}_{\varphi}\in \mathscr{T}(\alpha,\beta)$,$\varphi\in L^2$. Then $A^{\alpha, \beta}_{\varphi}=0 $ if and only if $\varphi\in  \overline{\alpha H^2}+\beta H^2$.
%\end{cor}

To give a proof of Theorem \ref{zero} we use the so-called Crofoot transform. For any inner function $\alpha$ and $w\in\mathbb{D}$, the Crofoot transform $J_w^{\alpha}$ is the multiplication operator given by
\begin{equation}\label{crof1}
J_w^{\alpha}f(z)=\frac{\sqrt{1-|w|^2}}{1-\overline{w}\alpha(z)}f(z).
\end{equation}

The Crofoot transform $J_w^{\alpha}$ is a unitary operator from $K_{\alpha}$ onto $K_{\alpha_w}$, where

\begin{equation}\label{crof2}
\alpha_w(z)=\frac{w-\alpha(z)}{1-\overline{w}\alpha(z)}.
\end{equation}
(see, for example, \cite[Thm. 10]{crof} and \cite[pp. 521--523]{s}). Moreover,
\begin{displaymath}
\begin{split}
\left(J_w^{\alpha}\right)^{*}f&=\left(J_w^{\alpha}\right)^{-1}f=J_w^{\alpha_w}f\\
&=\frac{\sqrt{1-|w|^2}}{1-\overline{w}\alpha_w}f=\frac{1-\overline{w}\alpha}{\sqrt{1-|w|^2}}f.
\end{split}
\end{displaymath}

%To prove Proposition \ref{pp1} we need a technical lemma.
\begin{lem}
Let $\alpha$ be an inner function and $w\in\mathbb{D}$. For every $z\in\mathbb{D}$ we have
\begin{equation}\label{kernel}
k_z^{\alpha_w}=\frac{1-|w|^2}{(1-w\overline{\alpha(z)})(1-\overline{w}\alpha)}k_z^{\alpha}.
\end{equation}
\end{lem}
\begin{proof}

Fix $w,z \in\mathbb{D}$. The reproducing kernel $k_z^{\alpha_w}$ is given by
\begin{displaymath}
k_z^{\alpha_w}(\lambda)=\frac{1-\overline{\alpha_w(z)}\alpha_w(\lambda)}{1-\overline{z}\lambda},\quad \lambda \in \mathbb{D}.
\end{displaymath}

Since
\begin{displaymath}
\begin{split}
1-\overline{\alpha_w(z)}\alpha_w(\lambda)&=1-\frac{\overline{w}-\overline{\alpha(z)}}{1-w\overline{\alpha(z)}}\frac{w-\alpha(\lambda)}{1-\overline{w}\alpha(\lambda)}\\ &=
\frac{(1-|w|^2)(1-\overline{\alpha(z)}\alpha(\lambda))}{(1-w\overline{\alpha(z)} )(1-\overline{w}\alpha(\lambda))},
\end{split}
\end{displaymath}
we have
\begin{displaymath}
\begin{split}
k_z^{\alpha_w}(\lambda)&= \frac{1-|w|^2}{(1-w\overline{\alpha(z)})(1-\overline{w}\alpha(\lambda))}\frac{1-\overline{\alpha(z)}\alpha(\lambda)}{1-\overline{z}\lambda}\\ &=
\frac{(1-|w|^2)}{(1-w\overline{\alpha(z)})(1-\overline{w}\alpha(\lambda))}k_z^{\alpha}(\lambda).
\end{split}
\end{displaymath}
\end{proof}

It is known that the map
$$A\mapsto J_w^{\alpha}A\left(J_w^{\alpha}\right)^{-1}, \quad A\in\mathscr{T}(\alpha),$$
carries $\mathscr{T}(\alpha)$ onto $\mathscr{T}(\alpha_w)$ (see \cite{si}). A similar result is true for the asymmetric truncated Toeplitz operators.

\begin{prop}\label{pp1}
Let $\alpha$, $\beta$ be two inner functions. Let $a, b\in\mathbb{D}$ and let the functions $\alpha_a$, $\beta_b$ and the operators $J_a^{\alpha}\ \colon\ K_{\alpha}\rightarrow K_{\alpha_a}$, $J_b^{\beta}\ \colon\ K_{\beta}\rightarrow K_{\beta_b}$ be defined as in \eqref{crof2} and \eqref{crof1}, respectively. If $A$ is a bounded linear operator from $K_{\alpha}$ into $K_{\beta}$, then $A$ belongs to $\mathscr{T}(\alpha,\beta)$ if and only if $J_b^{\beta}A\left(J_a^{\alpha}\right)^{-1}$ belongs to $\mathscr{T}(\alpha_a,\beta_b)$. Moreover, if $A=A_{\varphi}^{\alpha,\beta}$, then $J_b^{\beta}A\left(J_a^{\alpha}\right)^{-1}=A_{\phi}^{\alpha_a,\beta_b}$ with
\begin{equation}\label{symbol}
\phi=\frac{(1-\overline{a}\alpha)(1-b\overline{\beta})}{\sqrt{1-|a|^2}\sqrt{1-|b|^2}}\varphi.
\end{equation}
%$$\phi=\frac{1-\overline{a}\alpha}{\sqrt{1-|a|^2}}\frac{1-b\overline{\beta}}{\sqrt{1-|b|^2}}\varphi.$$
\end{prop}
\begin{proof}

Let $A$ be a bounded linear operator from $K_{\alpha}$ into $K_{\beta}$. Assume first that $A$ belongs to $\mathscr{T}(\alpha,\beta)$, $A=A_{\varphi}^{\alpha,\beta}$ for $\varphi\in L^2(\partial\mathbb{D})$. We show that $J_b^{\beta}A\left(J_a^{\alpha}\right)^{-1}=A_{\phi}^{\alpha_a,\beta_b}$ with $\phi$ as in \eqref{symbol}.

For every $f\in K_{\alpha_a}^{\infty}$ and $z\in \mathbb{D}$ we have
\begin{displaymath}
\begin{split}
J_b^{\beta}A_{\varphi}^{\alpha,\beta}\left(J_a^{\alpha}\right)^{-1}f(z)&=\frac{\sqrt{1-|b|^2}}{1-\overline{b}\beta(z)} P_{\beta}\left(\frac{1-\overline{a}\alpha}{\sqrt{1-|a|^2}}\varphi f\right)(z)\\&=\frac{\sqrt{1-|b|^2}}{1-\overline{b}\beta(z)}\left\langle  \frac{1-\overline{a}\alpha}{\sqrt{1-|a|^2}}\varphi f;k_z^{\beta}\right\rangle.
\end{split}
\end{displaymath}
By \eqref{kernel},
\begin{displaymath}
\begin{split}
J_b^{\beta}A_{\varphi}^{\alpha,\beta}\left(J_a^{\alpha}\right)^{-1}f(z)&=\frac{\sqrt{1-|b|^2}}{1-\overline{b}\beta(z)}\left\langle  \frac{1-\overline{a}\alpha}{\sqrt{1-|a|^2}}\varphi f;\frac{\left(1-b\overline{\beta(z)}\right )(1-\overline{b}\beta)}{1-|b|^2} k_z^{\beta_b}\right\rangle\\&= \left\langle\frac{1-b\overline{\beta}}{\sqrt{1-|b|^2}} \frac{1-\overline{a}\alpha}{\sqrt{1-|a|^2}}\varphi f;k_z^{\beta_b} \right\rangle\\&=
P_{\beta_b}\left (\frac{(1-b\overline{\beta})(1-\overline{a}\alpha)}{\sqrt{1-|b|^2}\sqrt{1-|a|^2}}\varphi  f\right )(z)\\&=A_{\phi}^{\alpha_a,\beta_b}f(z) .
\end{split}
\end{displaymath}
Thus $A\in\mathscr{T}(\alpha,\beta)$ implies that $J_b^{\beta}A\left(J_a^{\alpha}\right)^{-1}\in\mathscr{T}(\alpha_a,\beta_b)$.

To prove the other implication assume that $J_b^{\beta}A\left(J_a^{\alpha}\right)^{-1}=A_{\phi}^{\alpha_a,\beta_b}\in  \mathscr{T}(\alpha_a,\beta_b) $ for some $\phi\in L^2(\partial\mathbb{D})$. Then
$$A=(J_b^{\beta})^{-1}A_{\phi}^{\alpha_a,\beta_b}J_a^{\alpha}=J_b^{\beta_b}A_{\phi}^{\alpha_a,\beta_b}\left(J_a^{\alpha_a}\right)^{-1}.$$
But $(\alpha_a)_a=\alpha$ and $(\beta_b)_b=\beta$, and, by the first part of the proof, $$A=J_b^{\beta_b}A_{\phi}^{\alpha_a,\beta_b}\left(J_a^{\alpha_a}\right)^{-1}=A_{\varphi}^{\alpha,\beta}$$ with

\begin{displaymath}
\varphi=\frac{(1-\overline{a}\alpha_a)(1-b\overline{\beta_b})}{\sqrt{1-|a|^2}\sqrt{1-|b|^2}}\phi.
\end{displaymath}
Hence, $A\in\mathscr{T}(\alpha,\beta)$. An easy computation shows that $\phi$ satisfies \eqref{symbol}.
\end{proof}

%We are now ready to prove the following.

\begin{proof}[Proof of Theorem \ref{zero}]

The fact that $\varphi\in \overline{\alpha H^2}+\beta H^2$ implies $A^{\alpha, \beta}_{\varphi}=0$  was established in the proof of Proposition \ref{zero1}. Assume now that $\varphi\in L^2(\partial\mathbb{D})$ and
$A^{\alpha, \beta}_{\varphi}=0$ .

If $\alpha(0)=\beta(0)=0$, then $\varphi\in \overline{\alpha H^2}+\beta H^2$ by Proposition \ref{zero1}. If $\alpha(0)\neq 0$ or $\beta(0)\neq 0$, put $a=\alpha(0)$, $b=\beta(0)$. By Proposition \ref{pp1},

\begin{displaymath}
0=J_b^{\beta}A^{\alpha,\beta}_{\varphi}\left(J_a^{\alpha}\right)^{-1}=A_{\phi}^{\alpha_a,\beta_b},
\end{displaymath} where
\begin{displaymath}
\phi=\frac{(1-\overline{a}\alpha)(1-b\overline{\beta})}{\sqrt{1-|a|^2}\sqrt{1-|b|^2}}\varphi.
\end{displaymath}

Since $\alpha_a(0)=\beta_b(0)=0$, by Proposition \ref{zero1},
\begin{displaymath}
\phi\in \overline{\alpha_a H^2}+\beta_b H^2.
\end{displaymath}
Therefore, there exist $h_1,h_2\in H^2$ such that
\begin{displaymath}
\frac{(1-\overline{a}\alpha)(1-b\overline{\beta})}{\sqrt{1-|a|^2}\sqrt{1-|b|^2}}\varphi=\frac{\overline{a}-\overline{\alpha}}{1-a\overline{\alpha}}\overline{h}_1+\frac{b-\beta}{1-\overline{b}\beta}h_2,
\end{displaymath}
and
\begin{displaymath}
\varphi=\frac{\overline{a}-\overline{\alpha}}{1-\overline{a}{\alpha}} \frac{\sqrt{1-|a|^2}\sqrt{1-|b|^2}}{(1-{a}\overline{\alpha})(1-b\overline{\beta})}\overline{h}_1+\frac{b-\beta}{1-{b}\overline{\beta}} \frac{\sqrt{1-|a|^2}\sqrt{1-|b|^2}}{(1-\overline{a}\alpha)(1-\overline{b}{\beta})}h_2.
\end{displaymath}
Since $|\alpha|=1$ and $|\beta|=1$ on the unit circle $\partial\mathbb{D}$, we see that
\begin{displaymath}
\frac{\overline{a}-\overline{\alpha}}{1-\overline{a}{\alpha}}=-\overline{\alpha}\quad\mathrm{and}\quad \frac{b-\beta}{1-{b}\overline{\beta}}=-\beta\quad \mathrm{on}\quad \partial\mathbb{D},
\end{displaymath}
and
\begin{displaymath}
\varphi=\overline{\alpha}\overline{g_1}+\beta g_2
\end{displaymath}
with

\begin{displaymath}
g_1=-\frac{\sqrt{1-|a|^2}\sqrt{1-|b|^2}}{(1-\overline{a}{\alpha})(1-\overline{b}{\beta})}h_1,\quad
g_2=\frac{\sqrt{1-|a|^2}\sqrt{1-|b|^2}}{(1-\overline{a}\alpha)(1-\overline{b}{\beta})}h_2,
\end{displaymath}
$g_1,g_2\in H^2$. This completes the proof.
\end{proof}

\begin{cor}\label{zero2}
If $\varphi$ is in $L^2(\partial\mathbb{D})$, then there is a pair of functions $\chi\in K_{\alpha}$, $\psi\in K_{\beta}$, such that $A_{\varphi}^{\alpha, \beta}=A_{\overline{\chi}+\psi}^{\alpha, \beta}$. If $\chi$, $\psi$ is one such pair, then the most general such pair is of the form $\chi-\overline{c}k_0^{\alpha}$, $\psi+c k_0^{\beta}$, with $c$ a scalar.
\end{cor}
\begin{proof}
The proof is analogous to the proofs given in \cite{s} and \cite{ptak}.

The function $\varphi\in L^2(\partial\mathbb{D})$ can be written as $\varphi=\varphi_{+}+\varphi_{-}$ with $\varphi_{+}, \overline{\varphi}_{-}\in H^2$. If $\chi=P_{\alpha}(\overline{\varphi}_{-})$ and $\psi=P_{\beta}(\varphi_{+})$, then $\varphi-\overline{\chi}-\psi\in\overline{\alpha H^2}+\beta H^2$. By Theorem \ref{zero}, $A_{\varphi}^{\alpha, \beta}=A_{\overline{\chi}+\psi}^{\alpha, \beta}$.

Note that for $f\in K_{\alpha}$,
\begin{displaymath}
%\begin{split}
A_{k_0^{\beta}}^{\alpha, \beta}f=P_{\beta}\left(f-\overline{\beta(0)}\beta f \right )=P_{\beta}f=A_1^{\alpha, \beta} f.
%\end{split}
\end{displaymath}
Since $\overline{\alpha}f\perp H^2$ for $f\in K_{\alpha} $, we get
\begin{displaymath}
%\begin{split}
A_{\overline{k}_0^{\alpha}}^{\alpha, \beta}f=P_{\beta}\left (f-\alpha(0)\overline{\alpha}f\right )=P_{\beta}f=A_1^{\alpha, \beta} f.
%\end{split}
\end{displaymath}
Therefore, if $A_{\varphi}^{\alpha, \beta}=A_{\overline{\chi}+\psi}^{\alpha, \beta}$ with $\chi \in K_{\alpha},\psi\in K_{\beta}$ as above and $\chi_1=\chi-\overline{c}k_0^{\alpha}$, $\psi_1=\psi+c k_0^{\beta}$ for some constant $c\in \mathbb{C}$, then
\begin{displaymath}
A_{\overline{\chi}_1+\psi_1}^{\alpha, \beta}=A_{\overline{\chi}}^{\alpha, \beta}-cA_{1}^{\alpha, \beta}+A_{\psi}^{\alpha, \beta}+cA_{1}^{\alpha, \beta}=A_{\varphi}^{\alpha, \beta}.
\end{displaymath}
Moreover, if $A_{\varphi}^{\alpha, \beta}=A_{\overline{\chi}+\psi}^{\alpha, \beta}=A_{\overline{\chi}_1+\psi_1}^{\alpha, \beta}$ for any other $\chi_1 \in K_{\alpha},\psi_1\in K_{\beta}$, then, by Theorem \ref{zero}, there exist $h_1,h_2\in H^2$ such that
\begin{displaymath}
\overline{\chi}+\psi-\overline{\chi}_1-\psi_1=\overline{\alpha h_1}+\beta h_2.
\end{displaymath}
Hence
\begin{displaymath}
\psi-\psi_1=\beta h_2+\overline{\alpha h_1+\chi_1-\chi}
\end{displaymath} and
\begin{displaymath}
%\begin{split}
\psi-\psi_1=P_{\beta}(\psi-\psi_1)= P_{\beta}(\overline{\alpha h_1+\chi_1-\chi})=c_1P_{\beta}1=c_1k_0^{\beta}
%\end{split}
\end{displaymath}
for some constant $c_1$. Similarly,
\begin{displaymath}
\chi-\chi_1=\alpha h_1+\overline{\beta h_2+\psi_1-\psi}
\end{displaymath}
and
\begin{displaymath}
%\begin{split}
\chi-\chi_1=P_{\alpha}(\chi-\chi_1)=P_{\alpha}(\overline{\beta h_2+\psi_1-\psi})=c_2k_0^{\alpha}
%\end{split}
\end{displaymath}
for some constant $c_2$.

From this,
\begin{displaymath}
\begin{split}
0=A_{\overline{\chi}-\overline{\chi}_1+\psi-\psi_1}^{\alpha, \beta}&=\overline{c}_2A_{\overline{k}_0^{\alpha}}^{\alpha, \beta}+c_1A_{k_0^{\beta}}^{\alpha, \beta}\\&= (\overline{c}_2+c_1)A_1^{\alpha, \beta}=
(\overline{c}_2+c_1)P_{\beta|K_{\alpha}}.
\end{split}
\end{displaymath}
By Lemma \ref{le1}, $\overline{c}_2+c_1=0$. Putting $c=-c_1=\overline{c}_2$ we have $\psi_1=\psi+ck_0^{\beta}$
and $\chi_1=\chi-\overline{c}k_0^{\alpha}$.
\end{proof}

\section{Rank-one operators in $\mathscr{T}(\alpha,\beta)$}

Recall, that the model space $K_{\alpha}$ is equipped with a natural conjugation
(antilinear, isometric involution) $C_{\alpha}\colon
K_{\alpha}\rightarrow K_{\alpha}$, defined in terms of the boundary
values by
\begin{equation*}%\label{numerek3}
C_{\alpha}f(z)=\alpha(z)\overline{z}\overline{f(z)},\quad |z|=1
\end{equation*}
(see \cite[Subsection 2.3]{s}, for more details). %Actually, it can be
%verified that \eqref{numerek3} defines an antilinear and isometric
%involution on $L^2(\partial\mathbb{D})$, which maps $\alpha H^2$
%onto $\overline{H_0^2}$, $\overline{H_0^2}$ onto $\alpha H^2$, and
%preserves $K_{\alpha}$.
%For $f\in K_{\alpha}$ we will write $\widetilde{f}$ in place of
%$C_{\alpha}f$ when no confusion can arise.
A short calculation shows that the conjugate kernel
$\widetilde{k}_{w}^{\alpha}=C_{\alpha}k_{w}^{\alpha}$ is given by
$$\widetilde{k}_{w}^{\alpha}(z)=\frac{\alpha(z)-\alpha(w)}{z-w}.$$
If $\eta\in\partial\mathbb{D}$ and $k_{\eta}^{\alpha}\in K_{\alpha}$, then
$$\widetilde{k}_{\eta}^{\alpha}(z)=\frac{\alpha(z)-\alpha(\eta)}{z-\eta}=\alpha(\eta)\overline{\eta}k_{\eta}^{\alpha}(z).$$

We can now give some examples of rank-one asymmetric truncated Toeplitz operators (compare with \cite[Thm. 5.1]{s}).

\begin{prop}\label{rankone}
Let $ \alpha$, $\beta$ be two nonconstant inner functions.
\begin{itemize}
\item[(a)] For $w\in \mathbb{D}$, the operators $\widetilde{k}_w^{\beta}\otimes {k}_w^{\alpha}$ and
$k_w^{\beta}\otimes \widetilde{k}_w^{\alpha}$ belong to $\mathscr{T}(\alpha,\beta)$, $$\widetilde{k}_w^{\beta}\otimes {k}_w^{\alpha}=A^{\alpha, \beta}_{\frac{{\beta(z)}}{{z}-{w}}}\quad \mathrm{and}\quad k_w^{\beta}\otimes \widetilde{k}_w^{\alpha}=A^{\alpha, \beta}_{\frac{\overline{\alpha(z)}}{\overline{z}-\overline{w}}}.$$
\item[(b)] If both $\alpha$ and $\beta$ have an ADC at the point $\eta$ of $\partial\mathbb{D}$, then the operator $k_{\eta}^{\beta}\otimes {k}_{\eta}^{\alpha} $ belongs to $\mathscr{T}(\alpha,\beta)$, $$k_{\eta}^{\beta}\otimes {k}_{\eta}^{\alpha} =A^{\alpha, \beta}_{k_{\eta}^{\beta}+\overline{{k}}_{\eta}^{\alpha}-1}.$$
\end{itemize}
\end{prop}
\begin{proof}
$(\mathrm{a})$ Let $w\in \mathbb{D}$ and $f\in K_{\alpha}$. Since $\frac{f(z)-f(w)}{z-w}\in K_{\alpha}$ (\cite[Subsection 2.6]{s}), we have
\begin{displaymath}
\begin{split}
A_{\frac{{\beta(z)}}{{z}-{w}}}^{\alpha, \beta}f&=P_{\beta}\left(\frac{{\beta(z)}}{{z}-{w}}f(z)\right)\\
&=P_{\beta}\left(\beta(z)\frac{f(z)-f(w)}{z-w}+f(w)\frac{{\beta(z)-\beta(w)}}{{z}-{w}}+f(w)\frac{{\beta(w)}}{{z}-{w}}\right)\\
&= f(w)P_{\beta}\left(\frac{\beta(z)-\beta(w)}{z-w}\right)+f(w)\beta(w)P_{\beta}\left(\frac{\overline{z}}{1-w\overline{z}}\right)\\
&=f(w)\widetilde{k}_w^{\beta}=\langle f,k_w^{\alpha}\rangle\widetilde{k}_w^{\beta}=\widetilde{k}_w^{\beta}\otimes {k}_w^{\alpha}(f).
\end{split}
\end{displaymath}
Similarly,
\begin{displaymath}
\begin{split}
A_{\frac{\overline{\alpha(z)}}{\overline{z}-\overline{w}}}^{\alpha, \beta}f&=P_{\beta}\left(\frac{\overline{\alpha(z)}}{\overline{z}-\overline{w}}f(z)\right)=P_{\beta}\left(\overline{z}\frac{\overline{\alpha}(z)zf(z)}{\overline{z}-\overline{w}}\right)=P_{\beta}\left(\overline{z}\frac{\overline{C_{\alpha}f(z)}}{\overline{z}-\overline{w}}\right)\\&= P_{\beta}\left(\overline{z}\frac{\overline{C_{\alpha}f(z)}-\overline{C_{\alpha}f(w)}}{\overline{z}-\overline{w}}+\overline{z}\frac{\overline{C_{\alpha}f(w)}}{\overline{z}-\overline{w}}\right)=
\overline{C_{\alpha}f(w)}P_{\beta}(k_w)\\
&=\overline{C_{\alpha}f(w)}k_w^{\beta}=\overline{\langle C_{\alpha}f,k_w^{\alpha}\rangle}k_w^{\beta}=\langle f,\widetilde{k}_w^{\alpha}\rangle k_w^{\beta}=k_w^{\beta}\otimes \widetilde{k}_w^{\alpha}(f).
\end{split}
\end{displaymath}
\vspace{0.1cm}
$(\mathrm{b})$ Let $w\in \mathbb{D}$. Then
$$A_{k_w}^{\alpha, \beta}=A_{k_w^{\beta}}^{\alpha, \beta}\quad\mathrm{and}\quad A_{\overline{k}_w}^{\alpha, \beta}=A_{\overline{k}_w^{\alpha}}^{\alpha, \beta}.$$
Indeed,
\begin{displaymath}
%\begin{split}
A_{k_w^{\beta}}^{\alpha, \beta}f=P_{\beta}\left ((1-\overline{\beta(w)}\beta)k_wf\right )=
P_{\beta}\left(k_wf\right)=A_{k_w}^{\alpha, \beta}f,
%\end{split}
\end{displaymath}
for every $f\in K_{\alpha}$. From this,
\begin{displaymath}
%\begin{split}
A_{\overline{k}_w^{\alpha}}^{\alpha, \beta}= \left(A_{{k}_w^{\alpha}}^{\beta, \alpha} \right)^{*}=\left(A_{{k}_w}^{\beta, \alpha} \right)^{*}=A_{\overline{k}_w}^{\alpha, \beta}.
%\end{split}
\end{displaymath}
Since for $w\neq 0$ and $|z|=1$,
\begin{displaymath}
\begin{split}
\frac{\beta(z)}{z-w}&=\frac{\beta(z)-\beta(w)}{z-w}+\frac{\beta(w)}{z-w}\\&=
\widetilde{k}_w^{\beta}(z)+\frac{\beta(w)}{w}\cdot\frac{w\overline{z}}{1-w\overline{z}}=
\widetilde{k}_w^{\beta}(z)+\frac{\beta(w)}{w}\left (\overline{k}_w(z)-1\right ),
\end{split}
\end{displaymath}
we have, by part $(\mathrm{a})$,
\begin{displaymath}
%\begin{split}
\widetilde{k}_w^{\beta}\otimes {k}_w^{\alpha}=A_{\frac{{\beta(z)}}{{z}-{w}}}^{\alpha, \beta}=A_{\widetilde{k}^{\beta}_w+\frac{\beta(w)}{w}\left(\overline{k}_w-1\right)}^{\alpha, \beta}=A_{\widetilde{k}^{\beta}_w+\frac{\beta(w)}{w}\left(\overline{k}^{\alpha}_w-k_0^{\beta}\right)}^{\alpha, \beta}.
%\end{split}
\end{displaymath}
%Similarly,
%\begin{displaymath}
%k_w^{\beta}\otimes \widetilde{k}_w^{\alpha}=A^{\alpha, \beta}_{\overline{\widetilde{k}}^{\alpha}_w+\frac{\overline{\alpha(w)}}{w}\left(k_w^{\beta}-k_0^{\beta}\right)}.
%\end{displaymath}
If $\alpha$ and $\beta$ have an ${ADC}$  at $\eta\in \partial\mathbb{D}$, then $k_w^{\alpha}$ and $k_w^{\beta}$ converge in norm to $k_{\eta}^{\alpha}$ and $k_{\eta}^{\beta}$, respectively, as $w$ tends to $\eta$ nontangentially. Hence
$\widetilde{k}_w^{\beta}\otimes {k}_w^{\alpha}$ tends to $\widetilde{k}_{\eta}^{\beta}\otimes {k}_{\eta}^{\alpha}$ in the operator norm. On the other hand,
\begin{displaymath}
\widetilde{k}^{\beta}_w+\frac{\beta(w)}{w}\left(\overline{k}^{\alpha}_w-k_0^{\beta}\right)\longrightarrow \widetilde{k}^{\beta}_{\eta}+\frac{\beta(\eta)}{\eta}\left(\overline{k}^{\alpha}_{\eta}-k_0^{\beta}\right)\quad\mathrm{in}\quad L^2(\partial\mathbb{D}),
\end{displaymath}
which implies that
\begin{displaymath}
A^{\alpha, \beta}_{\widetilde{k}^{\beta}_w+\frac{\beta(w)}{w}\left(\overline{k}^{\alpha}_w-k_0^{\beta}\right)}f\longrightarrow A^{\alpha, \beta}_{\widetilde{k}^{\beta}_{\eta}+\frac{\beta(\eta)}{\eta}\left(\overline{k}^{\alpha}_{\eta}-k_0^{\beta}\right)}f\quad\mathrm{in}\quad H^2,
\end{displaymath}
for every $f\in K_{\alpha}^{\infty}$. Therefore,
\begin{displaymath}
\widetilde{k}_{\eta}^{\beta}\otimes {k}_{\eta}^{\alpha}=A_{\widetilde{k}_{\eta}^{\beta}+\frac{\beta(\eta)}{\eta}\left(\overline{k}_{\eta}^{\alpha}-k_0^{\beta}\right)}^{\alpha, \beta}.
\end{displaymath}
But
\begin{displaymath}
\widetilde{k}_{\eta}^{\beta}(z)=\frac{\beta(z)-\beta(\eta)}{z-\eta}=\frac{\beta(\eta)}{\eta} k_{\eta}^{\beta}(z),
\end{displaymath}
and
\begin{displaymath}
\begin{split}
k_{\eta}^{\beta}\otimes k_{\eta}^{\alpha}&= \frac{\eta}{\beta(\eta)}\widetilde{k}_{\eta}^{\beta}\otimes k_{\eta}^{\alpha}= \frac{\eta}{\beta(\eta)}A_{\frac{\beta(\eta)}{\eta}\left(k_{\eta}^{\beta}+\overline{k}_{\eta}^{\alpha}-k_0^{\beta}\right)}^{\alpha, \beta}\\&=
A_{k_{\eta}^{\beta}+\overline{k}_{\eta}^{\alpha}-k_0^{\beta}}^{\alpha, \beta}=A_{k_{\eta}^{\beta}+\overline{k}_{\eta}^{\alpha}-1}^{\alpha, \beta}.
\end{split}
\end{displaymath}
\end{proof}

%It was proved in \cite[Thm. 5.1]{s} that the only rank-one operators in $\mathcal{T}(\alpha)$ are the nonzero scalar multiples of the operators $\widetilde{k}_w^{\alpha}\otimes {k}_w^{\alpha}$, $k_w^{\alpha}\otimes \widetilde{k}_w^{\alpha}$ and $k_{\eta}^{\alpha}\otimes {k}_{\eta}^{\alpha}$. It is still an open question whether the scalar multiples of the operators from Proposition \ref{rankone} are the only rank-one operators in $\mathcal{T}(\alpha,\beta)$ for arbitrary inner functions $\alpha$ and $\beta$.

\end{document}